\documentclass[12pt]{amsart}

\usepackage{amsmath, amssymb, amsthm}
\usepackage{enumerate,hyperref}
\hypersetup{
    pdftitle=   {Number of cliques in graphs with a forbidden subdivision},
   pdfauthor=  {Choongbum Lee, Sang-il Oum}
}

\theoremstyle{plain}
\newtheorem{thm}{Theorem}[section]
\newtheorem*{thm*}{Theorem}
\newtheorem{prop}[thm]{Proposition}
\newtheorem{lem}[thm]{Lemma}
\newtheorem*{lem*}{Lemma}

\theoremstyle{remark}

\newcommand\abs[1]{\lvert #1\rvert}

\title{Number of cliques in graphs with a forbidden subdivision}
\author{Choongbum Lee}
\address{Department of Mathematics, MIT, Cambridge, MA 02139-4307, USA} 
\email{cb\_lee@math.mit.edu}
\thanks{C.~L.~is supported by NSF Grant DMS-1362326.}
\author{Sang-il Oum}
\address{Department of Mathematical Sciences, KAIST, Daejeon, 34141,
  South Korea}
\email{sangil@kaist.edu}
\thanks{S.~O.~is supported by Basic Science Research
  Program through the National Research Foundation of Korea (NRF)
  funded by  the Ministry of Science, ICT \& Future Planning
  (2011-0011653).}
\date\today

\begin{document}

\begin{abstract}
We prove that for all positive integers $t$, every $n$-vertex graph 
with no $K_t$-subdivision has at most $2^{50t}n$ cliques.
We also prove that asymptotically, such graphs contain at most $2^{(5+o(1))t}n$ cliques, where $o(1)$ tends to zero as $t$ tends to infinity.
This strongly answers a question of D.~Wood %
asking if the number of cliques in $n$-vertex graphs with no $K_t$-minor
is at most $2^{ct}n$ for some constant $c$.

\end{abstract}

\maketitle

\section{Introduction}

A \emph{clique}  of a graph is a set of pairwise adjacent vertices.
 A graph $H$ is a \emph{minor} of a graph $G$ if $H$ can be formed from
$G$ by deleting edges and vertices and by contracting edges. An \emph{$H$-subdivision}
of a graph $G$ is a subgraph of $G$ that can be formed from an isomorphic copy of $H$ by replacing edges with
vertex-disjoint (non-trivial) paths.
Trivially, if a graph has an $H$-subdivision, then it has an $H$-minor. But the converse is not true in general.

The problem of determining the maximum number of edges in 
graphs with no $K_t$-minor or no $K_t$-subdivision is a
well-studied problem in extremal graph theory:
Kostochka~\cite{Kostochka} and Thomason~\cite{Thomason1984} proved that 
graphs with no $K_t$-minor have average degree at most $c t \sqrt{\ln t}$, 
and Bollob\'as and Thomason~\cite{BoTh}, and independently,
Koml\'os and Szemer\'edi~\cite{KoSz} proved that graphs 
with no $K_t$-subdivision have average degree at most $c't^2$,
where $c$ and $c'$ are some absolute constants not depending on $t$
(in fact, a theorem of Thomas and Wollan~\cite{TW2005} can be used to 
show that $c' \le 10$, see, \cite[Theorem 7.2.1]{Diestel2010}).
A graph is \emph{$d$-degenerate} if
all its induced subgraphs contain a vertex of degree at most $d$. 
The results mentioned above straightforwardly imply that
graphs with no $K_t$-minor are $ct\sqrt{\ln t}$-degenerate, 
and graphs with no $K_t$-subdivision are $c't^2$-degenerate.

We study a related problem of determining the maximum number of 
cliques in graphs with no $K_t$-minor or no $K_t$-subdivision. 
Our work can be viewed as an extension of 
Zykov's theorem~\cite{Zykov} that establishes a bound on the number of cliques in graphs with no $K_t$ subgraphs. 
For planar graphs, 
Papadimitriou and Yannakakis~\cite{PaYa81} and Storch~\cite{Storch2006} proved a linear upper bound and finally Wood~\cite{Wood} determined the exact upper bound $8n-16$ for $n$-vertex planar graphs.
Dujmovi\'c et al.~\cite{DFJSW2011} generalized this result to graphs on surfaces.

For graph with no $K_t$-minors, Reed and Wood~\cite{ReWo09} 
and Norine et al.~\cite{NSTW06} 
obtained an upper bound on the number of cliques  
by using the fact that an $n$-vertex $d$-degenerate graph with $n\ge d$ 
has at most $2^d (n - d + 1)$ cliques. By the results mentioned above, 
this implies that graphs 
with no $K_t$-minor have at most $2^{ct\sqrt{\ln t}}n$ cliques
and graphs with no $K_t$-subdivision have at most $2^{10t^2}n$ cliques.
Wood~\cite{Wood} then asked whether there exists a constant $c$ for which 
every $n$-vertex graph with no $K_t$-minor has at most $2^{ct} n$ cliques.
If true, then the bound would be best possible up to the constant $c$ 
in the exponent, since the $(t-2)$-th power of a path on $n$ vertices has no $K_t$-minor
and contains $2^{t-2} (n-t+3)$ cliques (including the empty set).
See Section~\ref{sec:remarks} for an alternative construction.

The results of Wood were later improved to 
$2^{c t \ln \ln t} n$ (for graphs with no $K_t$-minor)  and $2^{ct\ln t}$ (for graphs with no $K_t$-subdivision)
by Fomin, Oum, and Thilikos~\cite{FoOuTh}.
In this paper, we settle Wood's question by proving the bound not only for
graphs with no $K_t$-minor, but also for graphs with no $K_t$-subdivision.

\begin{thm} \label{thm:main}
  For all positive integers $t$, every $n$-vertex graph with no $K_t$-subdivision has at most $2^{50t} n$ cliques.
\end{thm}
Our proof also implies that such graphs have at most $2^{(5+o(1))t}n$ cliques.

\section{Proof of theorem}

One can enumerate all cliques of a given graph by choosing vertices one at a time, and
recursively exploring its neighbors. To be more precise, first choose a vertex $v_1$ 
of minimum degree and explore all cliques that contain $v_1$ by recursively 
applying the algorithm to the graph induced on 
the set $N(v_1)$. Once all cliques containing $v_1$ has been explored, remove $v_1$ from the graph,
choose a vertex $v_2$ of minimum degree in the remaining graph and repeat the algorithm. The algorithm enumerates 
each clique of the graph exactly once, since the $i$-th step of the algorithm 
enumerates all cliques that contain $v_i$ but do not contain any vertex from 
$\{v_1, \ldots, v_{i-1}\}$ (where $v_j$ is the vertex chosen at the $j$-th step).
We emphasize that we always choose the vertex of minimum degree within the remaining graph since
this choice blends particularly well with sparse graphs. This algorithm has been used
in various previous works (see e.g.~\cite{KlWi}).

This simple algorithm immediately implies a reasonable result.
Since $K_t$-minor free graphs are $c t \sqrt{\ln t}$-degenerate, the 
vertex $v_i$ chosen at the $i$-th step of the algorithm above will have
degree at most $c t \sqrt{\ln t}$ in the remaining graph at that time. 
Since the neighborhood of $v_i$ is $K_{t-1}$-minor free, the number of cliques 
added at the $i$-th step is at most 
\[
	\sum_{j=0}^{t-2} \binom{\lfloor c t \sqrt{\ln t}\rfloor}{j} 
	\le t(e c \sqrt {\ln t})^t
	\le 2^{c' t \ln \ln t},
\]
proving that $n$-vertex graphs with no $K_t$-minor have at most $2^{c' t \ln \ln t} n$ cliques.
One can similarly show that $n$-vertex graphs with no $K_t$-subdivision have at most $2^{c'' t \ln t}n$ cliques by using the following theorem mentioned in the introduction. 
Both of these bounds on the number of cliques were first proved in~\cite{FoOuTh} using a different argument.\footnote{The short proof presented in this paper is due to the second author and D.~Wood [private communication at the Barbados workshop on structural graph theory, Bellairs Research institute, 2013].}

\begin{thm}[{\cite[Theorem 7.2.1]{Diestel2010}}] \label{thm:extremal_subdivision}
For all $t \ge 1$, every graph of average degree at least $10t^2$ contains a $K_t$-subdivision.
\end{thm}

In this section, we show how a more detailed analysis of the algorithm gives
an improved bound on the number of cliques for graphs with no $K_t$-subdivision.

\subsection{Enumerating cliques}\label{subsec:clique}

The algorithm introduced above provides a natural tree structure, 
called the \emph{clique search tree}, to the cliques of a given graph $G=(V,E)$,
where each node of the tree corresponds to one step of exploration in the algorithm, 
and at the same time, one clique of the graph.
Formally, the clique search tree is a labelled tree 
defined as follows (since we are simultaneously considering
two graphs, we denote the vertices of $G$ by $v,w,\ldots,$ while we denote
the vertices of the tree by $a,b,\ldots$ and refer to them as nodes):

\begin{itemize}
  \setlength{\itemsep}{0pt} \setlength{\parskip}{.3em}
\item[1.] Start with a tree having a single node $a_0$ as a root node with label $L_{a_0} = V$.
\item[2.] Choose a leaf node $a$ of the current tree with $L_a \neq \emptyset$
and let $L := L_a$.
\begin{itemize}
  \setlength{\itemsep}{0pt} \setlength{\parskip}{.1em}
\item[2-1.] Choose a vertex $v \in L$ of minimum degree in $G[L]$.
\item[2-2.] Add a child node $b$ to $a$ in the tree and label it by the set $L_b = L \cap N(v)$.
\item[2-3.] Define $L \leftarrow L - \{v\}$.
\item[2-4.] Repeat Steps 2-1, 2-2 and 2-3 until $L = \emptyset$.
\end{itemize}
\item[3.] Repeat Step 2, until all leaves have label $\emptyset$.
\end{itemize}
Denote this tree as $T_G$. Thus $T_G$ is a rooted labelled tree, where each node $a$ is
labelled by some set $L_a \subseteq V(G)$ (distinct nodes might receive the same label). 
Note that the number of cliques in $G$ is exactly $\abs{V(T_G)}$, since there exists a one-to-one correspondence between nodes of $T_G$ and cliques of $G$. (The root node of $T_G$ corresponds to the empty set, which is also a clique by definition.) Hence to count cliques of $G$, it suffices to count nodes of $T_G$. 

The following proposition lists some 
useful properties of the tree $T_G$. A subtree $T'$ of $T_G$ is
a \emph{rooted subtree} if $T'$ contains the root node of $T_G$. The \emph{boundary nodes}
of a rooted subtree $T'$ is the set of nodes of $T'$ that are adjacent in $T_G$ to a node not in $T'$.

\begin{prop} \label{prop:tree_property}
If $G$ is a graph with no clique of size $t$, then the clique search tree $T_G$ has the following properties.
\begin{enumerate}[\rm (i)]
\item The number of nodes of $T_G$ is equal to the number of cliques of $G$. 
Moreover, for all non-negative integers $\ell$, the number of nodes of $T_G$ that are at distance 
exactly $\ell$ from the root node is equal to the number of cliques of $G$
of size $\ell$. 
\item For each node $a$ of $T_G$, the tree $T_{G[L_a]}$ is isomorphic (as a rooted
labelled tree) to the subtree of $T_G$ induced on $a$ and its descendents.
\item If $b$ is a descendent of $a$, then $L_b \subsetneq L_a$.
\item Let $T'$ be a rooted subtree of $T_G$ 
whose boundary nodes are all labelled by sets of size at most $m$. 
Then 
\[\abs{V(T_G)} \le \abs{V(T')} \cdot \sum_{i=0}^{t-1} \binom{m}{i} \le \abs{V(T')} 2^{m}.\]
\end{enumerate}
\end{prop}
\begin{proof}
Properties (i), (ii), and (iii) follow from the definition and the
discussions given above. To prove Property (iv),
suppose that we are given a tree $T' \subseteq T_G$. 
Since $T'$ is a rooted subtree, each node in $T_G$ is
either in $T'$ or is a descendant of a boundary node of $T'$.
Furthermore, by Properties (i) and (ii), each boundary node of $T'$ has at most
$\sum_{i=1}^{t-1} \binom{m}{i}$ descendants in $T_G$. Hence 
\[
	\abs{V(T_G)}
	\le \abs{V(T')} + \sum_{a \,: \, \text{boundary of } T'} \sum_{i=1}^{t-1} \binom{m}{i}
	\le \abs{V(T')} \cdot \sum_{i=0}^{t-1} \binom{m}{i}.
        \qedhere
\]
\end{proof}

\subsection{Graphs of large minimum degree}%

The simple argument given in the beginning of this section that proves the bound $2^{c't \ln \ln t} n$ for $K_t$-minor free graphs is equivalent to applying Proposition~\ref{prop:tree_property}~(iv) to the subtree induced on the root of $T_G$ and its children. Hence to improve 
on this bound, it would be useful to find a small rooted subtree $T' $ of $ T_G$ 
whose boundary nodes are all labelled by small sets. When does such a subtree exist?

A graph $G$ is called \emph{$(\beta, N)$-locally sparse} if every set $X$
of at least $N$ vertices has a vertex $v \in X$ of degree at most
$\beta \abs{X}$ in $G[X]$.\footnote{It is more common to define a $(\beta,N)$-locally sparse graph as a graph satisfying the following slightly stronger property: each subset $X$ of size at least $N$ contains at most $\beta \abs{X}^2$ edges.} This concept was first introduced by Kleitman and Winston~\cite{KlWi} 
in their study of the number of $C_4$-free graphs on $n$ vertices, and
has been successfully applied to several problems in extremal combinatorics.

In the following two lemmas, we utilize the concept of locally sparse graphs to 
handle a subcase of our theorem when the graph is small and dense.
This subcase turns out to be an important ingredient in the proof of general cases.

\begin{lem}\label{lem:subdiv}
  Let $G$ be an $m$-vertex graph with no $K_t$-subdivision. If $G$ has minimum degree at least $\frac{9}{10}m$, then $m \le \max\{\frac{20}{11}t, \frac{t^2}{5}\}$ and $G$ is $(1-\frac{m}{2t^2},\frac{20}{11}t)$-locally sparse.
\end{lem}

\begin{proof}
We may assume that $m \ge \frac{20}{11}t$, since otherwise the lemma is vacuously true.
Let $X$ be a subset of vertices of size $\abs{X}\ge \frac{20}{11}t$ and suppose that $G[X]$ has minimum degree at least $(1-\frac{m}{2t^2})\abs{X}$ (note that this quantity may be negative). If we sum $e(Y)$, the number of edges in $Y$, over all $t$-element subsets $Y$ of $X$, then each edge in $X$ is counted exactly $\binom{\abs{X}-2}{t-2}$ times. Therefore there exists a $t$-element subset $Y$ of $X$ such that 
  \begin{align*}
    e(Y) &\ge \frac{\binom{\abs{X}-2}{t-2} e(X)}  {\binom{\abs X}{t}}\\
    &=\binom{t}{2} \frac{e(X)}{\binom{\abs X}{2}}\\
    &> \binom{t}{2} \left(1 -\frac{m}{2t^2}\right)>\binom{t}{2}-\frac{m}{4}.
  \end{align*}
  
  Since every vertex of $G$ has degree at least $\frac{9}{10}m$, 
  every pair of vertices  of $G$  has at least $\frac{4}{5}m$ common neighbors. 
  For each non-edge $e = \{v, v'\}$ in $Y$, we can greedily 
  find a common neighbor  $w_e \in V(G)\setminus Y$ of $v$ and $v'$
  such that all chosen $w_e$  for all non-edges $e$ are distinct,
  because $Y$ has at most $\frac{m}{4}$ non-edges in $Y$
  and
  \[
  \frac{4}{5}m -\left( t+\frac{m}{4}\right)
  =\frac{11}{20}m-t \ge 0.
  \]
  Then $Y$ together with all chosen $w_e$ induces a $K_t$-subdivision in $G$, contradicting our assumption.
  Therefore, $G$ is $(1-\frac{m}{2t^2},\frac{20}{11}t)$-locally sparse. 
  Since $G$ has minimum degree at least $\frac{9}{10}m$, if $m \ge \frac{20}{11}t$, then this implies 
  \[
  	\frac{9}{10}m \le \left(1-\frac{m}{2t^2}\right)m,
  \]
  from which $m \le \frac{t^2}{5}$ follows.
\end{proof}

\begin{lem}\label{lem:subdivision_to_sum}
  Let $G$ be an $m$-vertex graph with no $K_t$-subdivision
  with $m\le \frac{t^2}{5}$. 
  If $G$ is $(1-\frac{m}{2t^2},\frac{20}{11}t)$-locally sparse,
  then $G$ contains less than $2^{5t}$ cliques.
\end{lem}
\begin{proof}
  If $m< 5t$, then trivially $G$ contains less than $2^{5t}$ cliques and therefore we may assume that $m\ge 5t$.
  Let $T_G$ be the clique search tree of $G$ and let $T'$ be the subtree of $T_G$ obtained by taking all nodes of distance at most $\lfloor 2\frac{t^2}{m}\ln \frac{m}{t}\rfloor$ from the root. 
  Then by the local sparsity condition, 
  the label set of each boundary node of $T'$ has cardinality less than $\max\{\frac{20}{11}t, \frac{10}{9}t\}=\frac{20}{11}t$, because 
  \[
  \left(1-\frac{m}{2t^2}\right)^{\lfloor 2\frac{t^2}{m}\ln \frac{m}{t}\rfloor } m
  < \frac{ e^{-\ln \frac{m}{t}} m}{ 1- \frac{m}{2t^2}} \le \frac{10}{9}t,
  \]
  where the last inequality follows from $m\le \frac{t^2}{5}$. By Proposition~\ref{prop:tree_property}~(i), the number of nodes of $T'$ is at most the number of cliques of size at most $\lfloor 2\frac{t^2}{m}\ln \frac{m}{t}\rfloor$ and so we have the following inequality:
  \[
  \abs{V(T')}\le \sum_{i=0}^{\lfloor 2\frac{t^2}{m}\ln \frac{m}{t}\rfloor} \binom{m}{i}.
  \]
  As $\frac{\ln x}{x}\le \frac{1}{e}$ for all $x>0$,  $  	2\frac{t^2}{m}\ln \frac{m}{t}\le \frac{2t}{e}<m$.
  Since $\sum_{i=0}^{\lfloor k\rfloor }\binom{m}{i}\le
\sum_{i=0}^{\lfloor k\rfloor  } \binom{m}{i} \left(\frac{k}{m}\right)^{i-k} 
\le \left(\frac{m}{k}\right)^k (1+\frac{k}{m})^n
\le \left( \frac{em}{k}\right)^k$ for all $k\le m$, we have 
\[
    \abs{V(T')}
    \le \left(\frac{em}{2\frac{t^2}{m}\ln \frac{m}{t}}\right)^{2\frac{t^2}{m}\ln \frac{m}{t}} 
    \le 
    \left(\frac{m^2}{t^2}\right)^{2\frac{t^2}{m} \ln \frac{m}{t}} 
    = e^{4t \frac{\ln^2(m/t)}{m/t}}.
  \]
  because  $2\ln \frac{m}{t}\ge 2\ln 5>e$.
  As $\frac{\ln^2 x}{x}\le \frac{4}{e^2}$ for all $x>1$, 
  \[
    \abs{V(T')} \le e^{\frac{16}{e^2} t} < 2^{3.13t}.
    \]
 Since the label set of each boundary node of $T'$ has cardinality less than $\frac{20}{11}t$, 
 by Proposition~\ref{prop:tree_property}~(iv),
 \[
	\abs{V(T_G)} \le \abs{V(T')} \cdot 2^{\frac{20}{11}t}<\abs{V(T')}2^{1.82t}.
 \]
 It follows that $G$ has at most $ 2^{(3.13+1.82)t}<2^{5t}$ cliques.
\end{proof}

\subsection{Finishing the proof}

In this subsection, we prove Theorem~\ref{thm:main}.

  We may assume that $t\ge 4$, because otherwise $G$ is a forest and
  contains at most $2n$ cliques.
  Given a graph $G$ with no $K_t$-subdivision, let $T_G$ be its clique search tree. 
  By Theorem \ref{thm:extremal_subdivision}, $G$ is $10t^2$-degenerate. Therefore
  every non-root node has a label set of cardinality at most $10t^2$,
  and thus has at most $10t^2$ children.
    
We construct a rooted subtree~$T'$ of the clique search tree $T_G$ according to the following recursive rule. 
First take the root node. Then for a node $a$ in $T'$,
take its child $a'$ to be in $T'$  if $\sqrt{10}t\le \abs{L_{a'}} < \frac{9}{10}\abs{L_a}$.
Since the label set of every non-root node has cardinality at most $10t^2$ and the 
cardinality of the label sets 
decrease by a factor of at least $\frac{9}{10}$ at each level, we see that
$T'$ is a tree of height at most $1 + \frac{\ln (10t^2)}{2\ln (10/9)}$.
Since the root of $T_G$ has exactly $n$ children, the number of nodes of $T'$ 
satisfies
\begin{align}
	\abs{V(T')}
	&\le n \cdot (10t^2)^\frac{\ln (10t^2)}{2\ln (10/9)}
	=n\cdot 2^
  \frac{ \ln^2 (10t^2)}{2\ln (10/9)\ln 2}\nonumber\\
&\label{eq:eq_3}\le n\cdot 2^  \frac{ t\ln^2 (160)}{8\ln (10/9)\ln 2}<2^{44.1t}n,
\end{align}
where the second to last inequality follows from the fact that $t \ge 4$
and $\frac{\ln^2( 10x^2)}{x}$ is decreasing for $x>\frac{e^2}{\sqrt{10}}$.

Further note that for each boundary node $a$ of $T'$, either $\abs{L_a}\le \sqrt{10}t$, or 
there exists a child $a'$ of $a$ for which $\abs{L_{a'}} \ge \frac{9}{10}\abs{L_a}$.
In the first case, the number of descendants of $a$ in $T_G$ is clearly at most $2^{\sqrt{10}t}$, which is less than $2^{5t}$.
In the latter case, let $v_1, v_2, \ldots, v_{\abs{L_a}}$ be the vertices in $L_a$ 
listed in the order that they were chosen by the algorithm, and let $a_1, a_2, \ldots, a_{|L_a|}$ 
be the corresponding nodes of $T_G$. Suppose that $i$ is the minimum index for which
$|L_{a_i}| \ge \frac{9}{10}\abs{L_a}$.
Define $X_a = \{v_i, v_{i+1}, \ldots, v_{\abs{L_a}}\}$ and
let $G_a = G[X_a]$. Notice that the clique search tree $T_{G_a}$ 
is isomorphic to the
subtree of $T_G$ induced on $a$, $a_i, \ldots, a_{\abs{L_a}}$, and 
the descendants of $a_i, a_{i+1}, \ldots, a_{\abs{L_a}}$ in $T_G$.
Hence, the total number of nodes of $T_G$ is at most 
\begin{align} 
	\abs{V(T_G)} 
	&\le \abs{V(T')} + \sum_{a \,:\, \text{boundary of } T'} (\abs{V(T_{G_a})} - 1) \nonumber \\
	&\le \abs{V(T')} \cdot \max_{a \,:\, \text{boundary of } T'} \abs{V(T_{G_a})}. \label{eq:eq_2}
\end{align}

By the definition of our algorithm, the vertex $v_i$ is a vertex of minimum degree
in the graph $G_a$, and hence
$G_a$ has minimum degree at least $|L_{a_i}| \ge \frac{9}{10}\abs{L_a} \ge \frac{9}{10}\abs{X_a}$.
By Lemma~\ref{lem:subdiv}, $G_a$ is $(1-\frac{|X_a|}{2t^2},\frac{20}{11}t)$-locally sparse
and $|X_a| \le \max\{\frac{20}{11}t, \frac{1}{5}t^2\}$. 
If $ |X_a| \le \frac{1}{5}t^2$, then $G_a$ satisfies the conditions of Lemma~\ref{lem:subdivision_to_sum},
and therefore the tree $T_{G_a}$ has at most $2^{5t}$ nodes.
Otherwise  $|X_a| \le \frac{20}{11}t$ and by Proposition~\ref{prop:tree_property}~(i), 
the tree $T_{G_a}$ has at most $2^{\frac{20}{11}t}$ nodes.
In either case, we have \[\abs{V(T_{G_a})} \le 2^{5t}.\] 
By substituting this bound
and \eqref{eq:eq_3} into \eqref{eq:eq_2}, we obtain the desired inequality
$\abs{V(T_G)} \le 2^{5t} \abs{V(T')} < 2^{(5+44.1)t}n<2^{50t}n$.

\section{Remarks} \label{sec:remarks}

In this paper, we proved Theorem \ref{thm:main} asserting that every 
$n$-vertex graph with no $K_t$-subdivision has at most $2^{50t}n$ cliques. 
In fact, our proof shows that such graphs have at most $2^{(5+o(1))t}n$ cliques, 
since \eqref{eq:eq_3} could have been replaced by the inequality $\abs{V(T')} \le 2^{o(t)} n$.

It remains to determine the best possible constants $c$ and $C$ for which
the number of cliques in an $n$-vertex graph with no $K_t$-subdivision is at most 
$2^{(c+o(1))t}n$ and at most $2^{Ct}n$. We showed that $c \le 5$ and $C\le 50$, while as mentioned in the
introduction, the $(t-2)$-th power of a path shows that $c \ge 1$.
Lemma~\ref{lem:subdiv} can be written as follows: 
if $G$ is an $m$-vertex $K_t$-subdivison-free graph
of minimum degree at least $(1 - \alpha)m$, then $m \le \max\{\frac{t}{1-2\alpha-\beta/2}, \frac{\alpha}{\beta}t^2\}$ and $G$ is $(1-\frac{\beta m}{t^2}, \frac{t}{1-2\alpha-\beta/2})$-locally sparse.
By taking $\alpha = 0.01$ and $\beta = 0.65$ and following an almost same proof, 
we can obtain $c<4$.
Similarly, by taking $\alpha=0.35$ and $\beta = 0.4$, we can obtain $C<20$.
(In the modified proof, when we compute an upper bound on the number of cliques in a graph on $\gamma t$ vertices,  
we may use the inequality $\sum_{i=0}^{t} \binom{\gamma t}{i}\le (\gamma e)^t$  instead of $2^{\gamma t}$ to achieve a better bound depending on $\gamma$.)

D.~Wood~\cite{Wood} showed that $c \ge \frac{2}{3} \log_2 3\approx 1.057$
because 
the complete $k$-partite graph $K_{2,2\ldots,2}$ contains $3^k$
cliques and has no $K_t$-subdivision for $t > \lfloor 3k/2 \rfloor$.

We remark that Kawarabayashi and Wood~\cite{KW2012a} proved that $n$-vertex graphs with no odd-$K_t$-minor have at most $O(n^2)$ cliques and unlike the case of graph minors, $n^2$ cannot be improved because $K_{n,n}$ has no odd-$K_3$-minor.

\medskip

\noindent \textbf{Acknowledgement.} We thank David Wood and the two anonymous 
referees for their valuable comments.

\end{document}